\documentclass{article}
\usepackage[ansinew]{inputenc} 
\usepackage{graphicx}
\usepackage{amsmath,amsthm,amssymb,amscd}
\usepackage{color}
\usepackage[colorlinks]{hyperref}

\title {A congruence property of irreducible\\ Laguerre polynomials in two variables}

\author{Nikolai A. Krylov and Zhangyuan Li}

\date {}

\begin{document}

\newtheorem{thm}{Theorem}
\newtheorem{lem}{Lemma}
\newtheorem{claim}{Claim}
\newtheorem{dfn}{Definition}
\newtheorem{cor}{Corollary}
\newtheorem{prop}{Proposition}
\newtheorem{example}{Example}

\def\natu 		{\mathbb N}
\def\inte 		{\mathbb Z}
\def\rati 		{\mathbb Q}
\def\real		{\mathbb R}
\def\GCD 		{{\rm gcd}}

\def\lla 		{\longleftarrow}
\def\lra 		{\longrightarrow}
\def\ra 		{\rightarrow}
\def\hra 		{\hookrightarrow}
\def\lmt 		{\longmapsto}

\maketitle

\parskip=3mm

\begin{abstract}
In this paper we introduce a version of irreducible Laguerre polynomials in two variables and prove for it a congruence property, 
which is similar to the one obtained by Carlitz for the classical Laguerre polynomials in one variable.
\end{abstract}


\section{Introduction}

The generalized Laguerre polynomials in one variable are defined for an arbitrary integer $n\geq 0$ and a parameter $\alpha > -1$ 
by Rodrigues' relation 
(see, for example, \cite{Dunkl}, \S 1.4.2)  
$$
L_n^{\alpha} (x) = \frac{1}{n!} e^xx^{-\alpha}\cdot D^n\left(e^{-x}x^{n+\alpha}\right),~~~~\mbox{where}~~~~D^n:=\frac{d^n}{dx^n}.
$$
In this paper we will consider only non-negative integer values for the parameter $\alpha$, i.e. we assume from now on that 
$\alpha \in \natu_0=\natu\cup\{0\}$. 
Expanding the definition of $L_n^{\alpha}(x)$ using the $n$-fold product rule and the Pochhammer symbol defined for all $x\in\real$ by
$$
(x)_0 = 1,~~~(x)_n = \prod\limits_{i=1}^n (x+i-1)~~~~\mbox{for all}~~n\in\natu,
$$
one immediately comes to the following explicit formulas (\cite{Dunkl}, \S 1.4.2):

\begin{equation}
\label{LaguerreX}
L_n^{\alpha} (x)  = \sum\limits_{j=0}^n\frac{(\alpha +1)_n\cdot (-n)_j}{(\alpha+1)_j \cdot n!}\cdot \frac{x^j}{j!} = \sum\limits_{j=0}^n\frac{(-1)^j}{j!}\cdot{n+\alpha \choose n - j}\cdot x^j
\end{equation}

It was established by Schur (\cite{Schur}) in 1929 that $L_n (x) = L_n^0 (x)$ are irreducible over the rationals for all $n\in\natu$. Recently
this result was generalized by 
Filaseta and Lam who proved that for all but finitely many $n\in\natu$, the polynomials $L_n^{\alpha}(x)$ where $\alpha$ 
is a rational number which is not a negative integer, are irreducible over $\rati$ (see \cite{Filaseta}). 
Note that reducible $L_n^{\alpha}$ do exist, for example, $L_2^2(x) = 1/2(x-2)(x-6)$. One of the key 
characteristics of the Laguerre polynomials $L_n^{\alpha}(x)$ (with a fixed $\alpha>-1$) is that they are orthogonal over the interval $(0,\infty)$ 
with respect to the weight function $\omega(x) = e^{-x}x^{\alpha}$ (see chapter 1 of \cite{Dunkl}). They also satisfy other interesting properties, including 
the one due to Carlitz (\cite{Carlitz}), who proved in 1954 that for all $n,m\in\natu$, and a rational number $\alpha$ that is integral $\pmod{m}$,

\begin{equation}
\label{Carlitz}
(n+m)!L_{n+m}^{\alpha}(x) \equiv n!L_n^{\alpha}(x)\cdot m!L_m^{\alpha}(x) \pmod{m}.
\end{equation}

There are various examples of families of orthogonal polynomials in several variables and certain properties of the following 
multivariable Laguerre polynomials have been studied in \cite{Dunkl} and \cite{Aktas}.

\begin{equation}
\label{LaguerreMR}
L_{n_1,\ldots,n_r}^{\alpha_1,\ldots,\alpha_r}(x_1,\ldots,x_r)=L_{n_1}^{\alpha_1}(x_1)\cdot L_{n_2}^{\alpha_2}(x_2) \cdot \ldots \cdot L_{n_r}^{\alpha_r}(x_r)
\end{equation}

\noindent Such multivariable Laguerre polynomials are orthogonal with respect to the weight function, which is the product of the corresponding 
weight functions $x_1^{\alpha_1}\cdot\ldots\cdot x_r^{\alpha_r}\cdot e^{-(x_1+\ldots + x_r)}$ 
over the domain, which is the cartesian product of the corresponding domains $\real^d_+=\{(x_1,\ldots,x_r)~|~0<x_j<\infty,~~j\in\{1,2,\ldots,r\}\}$ 
(see \cite{Aktas} and \cite{Dunkl}, \S 2.3.5). It is also clear from (\ref{LaguerreMR}) that such multiple Laguerre polynomials are 
reducible as soon as they have more than one variable.

In this paper, we introduce a version of two-variable Laguerre polynomials $L_{n,m}(x,y)$, which are irreducible over the rationals 
and prove that such Laguerre polynomials satisfy a congruence relation similar to (\ref{Carlitz}). 

The rest of this paper is divided up as follows. In \S2, we introduce our version 
of Laguerre polynomials in $x$ and $y$ using Rodrigues formula with partial derivatives and derive the corresponding 
explicit formulas similar to (\ref{LaguerreX}). In \S3, we establish several auxiliary lemmas and use them to give another proof of 
the congruence (\ref{Carlitz}) of Carlitz. \S 4 contains a proof of the corresponding congruence for two-variable Laguerre 
polynomials (see (\ref{Formula12}) below). In \S 5, we discuss the irreducibility over $\rati$. 
Other properties of $L_{n,m}(x,y)$, including the orthogonality, shall be discussed somewhere else.

\noindent {\bf Acknowledgement}: The authors gratefully acknowledge support from the Siena Summer Scholars program, 
that funds scholarly activities in which faculty members and students of Siena College collaborate during the summer. 
This article is a result of such collaboration.

\section{Laguerre Polynomials in two variables}

As we already wrote, the Laguerre polynomials in $x$ are defined for an arbitrary integer $n\geq 0$ and the parameter $\alpha = 0$ 
by Rodrigues' relation $L_n(x) = \frac{1}{n!} e^x\cdot D^n\left(e^{-x}x^n \right)$. We apply this approach to define the Laguerre polynomials in two variables as follows.

\begin{dfn}

For all $n,m\in \natu_0$ let
$$
L_{n,m}(x,y):=\frac{1}{n!\cdot m!} e^{(x+y)/2}\cdot D_{\partial}^{n+m}\left(e^{(-x-y)/2} x^ny^m\right),~~~\mbox{where}~~~D_{\partial}(f(x,y)):=f_x(x,y) + f_y(x,y).
$$

\end{dfn}

\noindent{\underline{Note}: {\it Since $e^{(x+y)/2}\cdot D_{\partial}\left(e^{(-x-y)/2}f(x,y)\right) = f_x(x,y) + f_y(x,y) - f(x,y)$, if it happens that $f(x,y)$ 
depends only on a single variable $x$ we obtain 
$$
e^{(x+y)/2}\cdot D_{\partial}\left(e^{(-x-y)/2}f(x,y)\right) = \frac{d}{dx}f(x) - f(x) = e^{x}\cdot D\left(e^{-x}f(x)\right),
$$
and hence naturally 
$$
L_{n,0}(x,y) = \frac{1}{n!\cdot 0!} e^{(x+y)/2}\cdot D_{\partial}^{n}\left(e^{(-x-y)/2} x^n y^0\right) = L_n(x),~~~~~ L_{0,m}(x,y) = L_m(y).
$$
By the same argument we also have 
$$
L_{n,m}(x,x) = \frac{1}{n!\cdot m!} e^{(x+y)/2}\cdot D_{\partial}^{n+m}\left(e^{(-x-y)/2} x^{n+m}\right) = \binom{n+m}{n} \cdot L_{n+m}(x).
$$}

\noindent Before giving the explicit formulas for $L_{n,m}(x,y)$ we prove the following formula.

\begin{lem}
For all $n,~m,~t\in\natu_0$ we have
\begin{equation}
\label{Formula1}
e^{(x+y)/2}\cdot D^t_{\partial}\left(e^{(-x-y)/2}x^n y^m \right) = \sum\limits_{i=0}^{\min(t,m)} \binom{t}{i}\cdot\frac{m!}{(m-i)!}\cdot 
\biggl(e^x D^{t-i}\left( e^{-x} x^n\right)\biggr)\cdot y^{m - i},
\end{equation}
where $D=\frac{d}{dx}$ is as in the definition of the classical Laguerre polynomials.
\end{lem}

\begin{proof}
We use induction on $t$ so suppose $t=0$. Hence $e^{(x+y)/2}\cdot D_{\partial}^{t}(e^{(-x-y)/2} x^n y^m) = x^n y^m$ and 
the R.H.S. of (\ref{Formula1}) also yields a single term $x^n y^m$. Now assume that (\ref{Formula1}) is true for $t=k-1$, then  
$$
e^{(x+y)/2}\cdot D^k_{\partial}\left(e^{(-x-y)/2}x^n y^m \right) = e^{(x+y)/2}\cdot D_{\partial}\left( e^{(-x-y)/2}\cdot e^{(x+y)/2} \cdot D^{k-1}_{\partial}\left(e^{(-x-y)/2}x^n y^m \right) \right) = 
$$
by induction hypothesis (and assuming for a moment that $k-1<m$)
$$
= e^{(x+y)/2}\cdot D_{\partial}\biggl(e^{(-x-y)/2}\biggl( e^x D^{k-1}\left( e^{-x} x^n\right) y^m + \cdots + \binom{k-1}{i}\frac{m!}{(m-i)!} e^x D^{k-1-i}\left( e^{-x} x^n\right)y^{m-i} + 
$$
\begin{equation}
\label{Formula2}
+ \binom{k-1}{i+1}\frac{m!}{(m - i - 1)!} e^x D^{k-1-i-1}\left( e^{-x} x^n\right)y^{m-i-1}  + \cdots +  \frac{m!}{(m-k+1)!}x^n y^{m-k+1} \biggr)\biggr). 
\end{equation}

\noindent Since $e^{(x+y)/2}\cdot D_{\partial}\left(e^{(-x-y)/2} g(x)y^k\right) =\bigl(g'(x)-g(x)\bigr)y^k + kg(x) y^{k -1} = 
e^x\cdot D(e^{-x}g(x)) y^k + k g(x) y^{k -1}$,
replacing each term in (\ref{Formula2}) by the corresponding two terms we can write 
$$
(\ref{Formula2}) = e^x \cdot D^k (e^{-x} x^n ) y^m + m e^x \cdot D^{k -1} (e^{-x} x^n ) y^{m-1} + \cdots +  \binom{k-1}{i}\frac{m!}{(m-i)!} e^x \cdot D^{k -i}(e^{-x} x^n) y^{m-i} + 
$$
$$
+ \binom{k-1}{i}\frac{m!}{(m-i -1)!} e^x \cdot D^{k -i - 1}(e^{-x} x^n) y^{m-i-1} +  \binom{k-1}{i+1}\frac{m!}{(m - i - 1)!} e^x D^{k -1 -i}\left( e^{-x} x^n\right)y^{m-i-1} + 
$$
$$
+  \binom{k-1}{i+1}\frac{m!}{(m - i - 2)!} e^x D^{k-1-i-1}\left( e^{-x} x^n\right)y^{m-i-2} + \cdots +  
$$
\begin{equation}
\label{Formula3}
 + \frac{m!}{(m-k+1)!}e^x D(e^{-x}x^n) y^{m-k+1}  +  \frac{m!}{(m-k)!} x^n y^{m-k}. 
\end{equation}
 
Notice that if $m=k-1$ then (\ref{Formula3}) ends with $ \frac{m!}{(m-k+1)!}e^x D(e^{-x}x^n) y^{m-k+1} = m! e^x D(e^{-x} x^n)$. Now, combining in (\ref{Formula3}) 
the coefficients of terms that have the same degree in $y$, introducing $j:=i+1$, and using the identity $\binom{a-1}{b}+\binom{a-1}{b+1}=\binom{a}{b+1}$ 
we obtain further that when $k\leq m$,
$$
(\ref{Formula3}) = \sum\limits_{j=0}^{k} \binom{k}{j}\cdot\frac{m!}{(m -j)!}\cdot \biggl(e^x D^{k - j}\left( e^{-x} x^n\right)\biggr)\cdot y^{m - j},
$$
which finishes the induction and proves lemma. If $k-1=m$ the alternation in formula is obvious.
\end{proof}

\noindent Next theorem gives the explicit formulas for $L_{n,m}(x,y)$ (cf. with (\ref{LaguerreX})).

\begin{thm}

For all $n,m\in\natu_0$ we have
$$
L_{n,m}(x,y)  = \sum\limits_{i=0}^m\frac{(-1)^i}{i!}\cdot{m+n\choose m - i}\cdot L_n^{i}(x)\cdot y^i = 
 \sum\limits_{s=0}^n\frac{(-1)^s}{s!}\cdot{n+m\choose n - s}\cdot L_m^{s}(y)\cdot x^s,~~~~~\mbox{and}
$$
\begin{equation}
\label{LagXY}
L_{n,m}(x,y)  = \sum\limits_{i=0}^m\sum\limits_{s=0}^n\frac{(-1)^{i+s}}{i!\cdot s!}\cdot{m+n\choose m - i}
\cdot{n+i\choose n - s}\cdot x^s \cdot y^i
\end{equation}
\end{thm}

\begin{proof} 
Using lemma 1 and formula (\ref{Formula1}) we can write 
$$
e^{(x+y)/2}\cdot D^{n+m}_{\partial}\left(e^{(-x-y)/2}x^n y^m \right) = \sum\limits_{j=0}^{m} \binom{m+n}{j}\cdot\frac{m!}{(m-j)!}\cdot \biggl(e^x D^{n+m-j}\left( e^{-x} x^n\right)\biggr)\cdot y^{m - j},
$$
which implies that for $i:=m-j\in\{0,\ldots,m\},~L_{n,m}(x,y)=$
\begin{equation}
\label{Formula4}
\frac{1}{n!\cdot m!}\cdot\sum_{i=0}^{m} \binom{m+n}{m-i}\cdot \frac{m!}{i!}\cdot \biggl(e^x D^{n+i}\left( e^{-x} x^n\right)\biggr)\cdot y^{i} 
=\sum_{i=0}^{m}\frac{1}{i!}\cdot \binom{m+n}{m-i}\cdot \biggl(\frac{1}{n!} e^x D^{n+i}\left( e^{-x} x^n\right)\biggr)\cdot y^{i}
\end{equation}
Since for Laguerre polynomials in a single variable $x$ (see \cite{Dunkl}, \S 1.4.2)
$$
\frac{d}{dx}\bigl(L_k^{\alpha}(x)\bigr) = -L_{k-1}^{\alpha+1}(x)~~~\mbox{and}~~~L_k^{\alpha}(x) = L_k^{\alpha+1}(x) -L_{k-1}^{\alpha+1}(x) 
\Rightarrow e^xD^t(e^{-x}L_k^{\alpha}(x)) = (-1)^t\cdot L_k^{\alpha+t}(x)
$$
we can continue formula (\ref{Formula4}) and write
$$
L_{n,m}(x,y)=\sum_{i=0}^{m} \frac{1}{i!}\cdot \binom{m+n}{m-i}\cdot \biggl(e^x D^i\bigl( e^{-x} L_n(x) \bigr)\biggr)\cdot y^{i} = 
\sum_{i=0}^{m} \frac{(-1)^i}{i!}\cdot \binom{m+n}{m-i}\cdot L_n^i (x) \cdot y^{i},
$$
which gives the first formula in (\ref{LagXY}). The second formula follows either from the symmetry or from an argument similar 
to the one we just gave. To obtain the third formula recall that by (\ref{LaguerreX}),
$$
L_n^{i} (x)  = \sum\limits_{s=0}^n\frac{(-1)^s}{s!}\cdot{n+i \choose n - s}\cdot x^s
$$
and hence
$$
L_{n,m}(x,y) = \sum_{i=0}^{m} \sum_{s=0}^{n}\frac{(-1)^{i+s}}{i!\cdot s!}\binom{m+n}{m-i}\binom{n+i}{n-s}\cdot x^{s}\cdot y^{i},
$$
as was required.
\end{proof}

It is well known (see \cite{Dunkl}, \S 1.4.2 and \S 2.3.5) that Laguerre polynomials in one variable satisfy the differential equation
\begin{equation}
\label{difeq}
x\cdot \frac{d^2}{dx^2}L_n^{\alpha}(x) + (\alpha + 1 -x)\cdot \frac{d}{dx}L_n^{\alpha}(x) + n\cdot L_n^{\alpha}(x) = 0,
\end{equation}
and the multiple Laguerre polynomials $L_{n_1,\ldots,n_r}^{\alpha_1,\ldots,\alpha_r}(x_1,\ldots,x_r)$ (recall (\ref{LaguerreMR}) above) 
satisfy the partial differential equation
$$
\sum_{i=1}^r x_i\cdot \frac{\partial^2}{\partial x_i^2}L_{n_1,\ldots,n_r}^{\alpha_1,\ldots,\alpha_r} + 
\sum_{i=1}^r \bigl((\alpha_i + 1 -x_i)\cdot \frac{\partial}{\partial x_i} 
L_{n_1,\ldots,n_r}^{\alpha_1,\ldots,\alpha_r}\bigr) + n\cdot L_{n_1,\ldots,n_r}^{\alpha_1,\ldots,\alpha_r} =0.
$$
Here is the corresponding analog for the Laguerre polynomial $L_{n,m}(x,y)$.

\begin{lem}
For all $n,m\in\natu_0,~L_{n,m}(x,y)$ satisfies the following system of partial differential equations. 
$$
\begin{pmatrix}
L_{xx} & L_{xy}\\
L_{yx} & L_{yy}\\
\end{pmatrix}\cdot
\begin{pmatrix}
x\\
y\end{pmatrix} + 
\begin{pmatrix}
L_{x} & 0\\
0 & L_{y}\\
\end{pmatrix}\cdot
\begin{pmatrix}
1 - x\\
1 - y\end{pmatrix} +
\begin{pmatrix}
L & 0\\
0 & L\\
\end{pmatrix}\cdot
\begin{pmatrix}
n\\
m\end{pmatrix} = 
\begin{pmatrix}
0\\
0\end{pmatrix},
$$
where we used the notations $L~\mbox{for}~L_{n,m}(x,y)$, $L_x~\mbox{for}~\frac{\partial}{\partial x}L_{n,m}(x,y)$, 
$L_{xy}~\mbox{for}~\frac{\partial^2}{\partial x\partial y}L_{n,m}(x,y)$, and so on.
\end{lem}
\begin{proof}
The proof is a straightforward computation using our first two explicit formulas in (\ref{LagXY}) and the differential equation (\ref{difeq}).
\end{proof}

\section{Another proof of the congruence of Carlitz}

In this technical section we prove several auxiliary results. Note that $(x)_n=x\cdot(x+1)\cdot\ldots\cdot(x+n-1)$ denotes the 
Pochhammer symbol, and $(p,q)$ stands for the greatest common divisor of $p$ and $q$.

\begin{lem}
For all $n,m\in\natu_0$, $t\in\{0,\ldots,n\}$, $l\in\{0,\ldots,m\}$, and non-zero $p,q\in\inte$ the following congruence holds.  
\begin{equation}
\label{cong1}
(l+1-p)_{m-l}\cdot (l+1+t+q)_{n-t}\equiv (l+1)_{m-l}\cdot (l+1+t)_{n-t}\pmod{(p,q)}
\end{equation} 

\end{lem}
\begin{proof}
We use induction on $n$. If $n=0$, then $t=0$, $(x)_0=1$ and the identity has the form 
$$
(l+1-p)_{m-l}\equiv(l+1)_{m-l}\pmod{(p,q)}
$$ 
But $l+k-p\equiv l+k \pmod{p}$, which implies that 
$$
(l+1-p)\cdot (l+2-p)\cdot \ldots\cdot (l+(m-l)-p) \equiv (l+1)\cdot (l+2)\cdot\ldots\cdot (l+(m-l)) \pmod{(p,q)}
$$
and proves the base of induction. 

Assume now that (\ref{cong1}) holds for $\forall n \in  \{0,\ldots ,k\}$ and let us prove it for $n=k+1$. If in this case 
$t=n=k+1$ then (\ref{cong1}) again has the form $(l+1-p)_{m-l}\equiv(l+1)_{m-l}\pmod{(p,q)}$ so we will assume 
till end of the proof that $t \in \{0,\ldots,k\}$. Then we can write the L.H.S of (\ref{cong1}) as  
$$
(l+1-p)_{m-l}\cdot (l+1+t+q)_{k+1-t}=(l+1-p)_{m-l}\cdot (l+1+t+q)_{k-t}\cdot(l+1+k+q),
$$
which is congruent modulo $q$ to $(l+1-p)_{m-l}\cdot(l+1+t+q)_{k-t}\cdot(l+1+k)$. Since by the induction hypothesis 
$$
(l+1-p)_{m-l}\cdot (l+1+t+q)_{k-t} \equiv (l+1)_{m-l} \cdot (l+1+t)_{k-t} \pmod{(p,q)}
$$
we deduce that 
$$
(l+1-p)_{m-l}\cdot (l+1+t+q)_{k+1-t}\equiv (l+1)_{m-l}\cdot (l+1+t)_{k-t}\cdot (l+1+k)\pmod{(p,q)},
$$ 
and since $(l+1+t)_{k-t}\cdot (l+1+k)=(l+1+t)_{k+1-t}$ the lemma is proved.
\end{proof}

Another way to prove it would be to notice, as we did in the base of induction, that $(l+1+t+q)_{n-t}\equiv (l+1+t)\pmod{q}$ 
and multiply two congruences modulo $p$ and modulo $q$ together to obtain the congruence modulo $(p,q)$. 
Our next statement gives a congruence relation for the 
generalized Laguerre polynomials $L_n^{\alpha}(x)$, which we will use later, and which is interesting in its own right.

\begin{prop}
For all $n,m\in \natu_0$, $q\in\natu$, $i \in \{q,\ldots, m+q\}$, and non-zero $p \in \inte$ we have
\begin{equation}
\label{cong2} 
(i-q+1-p)_{m-(i-q)}\cdot n!\cdot L_n^i(x) \equiv (i-q+1)_{m-(i-q)}\cdot n!\cdot L_n^{i-q}(x)\pmod{(p,q)}
\end{equation}
\end{prop}
\begin{proof}
If $n=0$, then $L_n^{i}(x)=L_n^{i-q}(x)=1$. If we define $l:=i-q$ then the identity we just saw above 
$(l+1-p)_{m-l} \equiv (l+1)_{m-l}\pmod{(p,q)}$, proves the statement in this case. If $n\geq 1$, to prove the 
proposition we compare the corresponding coefficients on the L.H.S and on the R.H.S of (\ref{cong2}). Since 
$$
L_n^{\alpha}(x)=\sum_{j=0}^{n}\frac{(-1)^j}{j!}\binom{n+\alpha}{n-j}\cdot x^j
$$
we see that the coefficients of $x^t$, $\forall t \in \{0,\ldots,n\}$ on the L.H.S. and on the R.H.S. will be respectively 
$$
(i-q+1-p)_{m-(i-q)}\cdot (-1)^{t}\cdot \frac{n!}{t!}\cdot \binom{n+i}{n-t}~~~~\mbox{and}~~~~
(i-q+1)_{m-(i-q)}\cdot (-1)^{t}\cdot \frac{n!}{t!}\cdot\binom{n+i-q}{n-t}.
$$

\noindent Since 
$$
\frac{n!}{t!}\cdot \binom{n+i}{n-t}=\binom{n}{t}\cdot (t+i+1)_{n-t} ~~~~\mbox{and}~~~~
\frac{n!}{t!}\cdot \binom{n+i-q}{n-t} = \binom{n}{t}\cdot (t+i-q+1)_{n-t},
$$
we can use the congruence, which follows from lemma 3 with $l=i-q$,
$$
(i+1-p-q)_{m-(i-q)}\cdot (i+1+t)_{n-t} \equiv (i+1-q)_{m-(i-q)}\cdot (i+1+t-q)_{n-t}\pmod{(p,q)},
$$
to obtain that 
$$
(i-q+1-p)_{m-(i-q)}\cdot (-1)^{t}\cdot \frac{n!}{t!}\cdot \binom{n+i}{n-t} = (-1)^t\cdot\binom{n}{t}\cdot(t+i+1)_{n-t}\cdot (i+1-p-q)_{m-(i-q)}
$$
is congruent modulo the $\GCD(p,q)$ to
$$
(-1)^t\cdot \binom{n}{t}\cdot (i+1+t-q)_{n-t}\cdot (i+1-q)_{m-(i-q)} = (-1)^t\cdot (i+1-q)_{m-(i-q)}\cdot \frac{n!}{t!}\cdot \binom{n+i-q}{n-t}.
$$
It shows that the coefficients of $x^t$ on the L.H.S. and on the R.H.S. of (\ref{cong2}) coincide modulo the $(p,q)$ and proves 
the proposition.
\end{proof}

\noindent Here is a cute congruence, which follows trivially from this proposition by taking $p=q$ and $i=m+q$.

\begin{cor}
For all $n,m \in\natu_0$ and $q\in \natu$, 
$$
n!L_n^{m+q}(x) \equiv n!L_n^{m} (x) \pmod{q}
$$
\end{cor}

The proposition will be used later in the proof of our main result. Let us continue with a few more congruences 
before proving the identity (\ref{Carlitz}) in a direct way (cf. with Theorem 3 of \cite{Carlitz}).  In the next section this identity 
together with the proof will be generalized to the case of Laguerre polynomials $L_{n,m}(x,y)$.

%\begin{lem}
%For all $n\in\natu_0$, $i,p\in\inte$, and $j\in\{0,\ldots,n\}$ we have 
%\begin{equation}
%\label{cong3}
%(j+1)_{n-j}\cdot (j+1+i)_{n-j}\equiv (j+1+p)_{n-j}\cdot (j+1+i+p)_{n-j} \pmod{p}.
%\end{equation}

%\end{lem}
%\begin{proof}
%Let's again use induction on $n \in \natu_0$. If $n=0\Rightarrow j=0$ and (\ref{cong3}) turns into an obviously true identity 
%$1\equiv 1\pmod{p}$. Next we assume that (\ref{cong3}) holds for $\forall n \in \{0,\ldots,k\}$ and prove it for $n=k+1$. 
%If in this case,  $j=n=k+1$, then again all four factors in (\ref{cong3}) become equal to 1 and the congruence is true. So 
%let's assume now that $n=k+1$ and $j \in \{0,\ldots,k\}$. Then we have
%$$
%(j+1+p)_{k+1-j}=(j+1+p)_{k-j}\cdot (k+1+p)~\mbox{and}~
%(j+1+i+p)_{k+1-j}=(j+1+i+p)_{k-j}\cdot (k+1+i+p).
%$$
%These identities give 
%$$
%(j+1+p)_{k+1-j}\cdot (j+1+i+p)_{k+1-j}=(j+1+p)_{k-j}\cdot (j+1+i+p)_{k-j}\cdot (k+1+p)\cdot (k+1+i+p)
%$$
%which is congruent modulo $p$ (since $a+p\equiv a \pmod{p}$) to 
%\begin{equation}
%\label{Formula5}
%(j+1+p)_{k-j}\cdot (j+1+i+p)_{k-j}\cdot (k+1)\cdot (k+1+i).
%\end{equation}
%By induction hypothesis the product of first two terms is congruent modulo $p$ to  $(j+1)_{k-j}\cdot (j+1+i)_{k-j}$, and since 
%$$
%(j+1)_{k+1-j}=(j+1)_{k-j}\cdot(k+1)~~~\mbox{and}~~~(j+1+i)_{k+1-j}=(j+1+i)_{k-j}(k+1+i)
%$$
%we see that (\ref{Formula5}) is congruent modulo $p$ to $(j+1)_{k+1-j}\cdot (j+1+i)_{k+1-j}$.
%This proves the identity (\ref{cong3}) for $n=k+1$ and finishes the induction proof.
%\end{proof}

\begin{lem}
For all $m,n \in \natu_0$, $q\in\natu$, $p \in \inte\setminus\{0\}$, and $i \in \{q,\ldots,m+q\}$ we have
\begin{equation}
\label{Formula6}
\binom{m}{m+q-i}\cdot (n+p+i-q+1)_{m-(i-q)}\equiv \binom{m+n}{m+q-i}\cdot (i-q-p+1)_{m-(i-q)}\pmod{(p,q)}
\end{equation}
\end{lem}
\begin{proof}
Using $a+b\equiv a\pmod{b},~\forall a\in\inte$ and $\forall b\in\natu$ deduce that 
$$
(n+p+i-q+1)_{m-(i-q)} = (n+p+i-q+1)\cdot\ldots\cdot (n+p+m)\equiv (n+i-q+1)\cdot\ldots\cdot (n+m) \pmod{p},
$$
$$
\mbox{and}~~(n+i-q+1)\cdot\ldots\cdot (n+m)\equiv (n+i+1)\cdot\ldots\cdot (n+m+q) = (n+i+1)_{m-(i-q)} \pmod{q},
$$
which imply that the L.H.S. of (\ref{Formula6}) 
$$
\binom{m}{m+q-i}\cdot (n+p+i-q+1)_{m-(i-q)}\equiv \binom{m}{m+q-i}\cdot (n+i+1)_{m-(i-q)} \pmod{(p,q)}.
$$
Similarly we have for the R.H.S. of (\ref{Formula6}) 
$$
\binom{m+n}{m+q-i}\cdot (i-q-p+1)_{m-(i-q)} \equiv \binom{m+n}{m+q-i}\cdot (i+1)_{m-(i-q)}\pmod{(p,q)},
$$
and therefore (\ref{Formula6}) will follow from 
\begin{equation}
\label{Formula7}
\binom{m}{m+q-i}\cdot (n+i+1)_{m-(i-q)} \equiv \binom{m+n}{m+q-i}\cdot (i+1)_{m-(i-q)}\pmod{(p,q)}.
\end{equation}
Furthermore, 
$$
\binom{m}{m+q-i}\cdot (n+i+1)_{m-(i-q)} = \frac{m!\cdot (n+m+q)!}{(i-q)!\cdot(m+q-i)!\cdot (n+i)!} = 
\binom{n+m+q}{m+q-i}\cdot\frac{m!}{(i-q)!}
$$
and
$$
\binom{m+n}{m+q-i}\cdot (i+1)_{m-(i-q)} = \frac{(m+n)!\cdot (m+q)!}{(n+i-q)!\cdot(m+q-i)!\cdot i!} = 
(n+i-q+1)_{m-(i-q)}\cdot\binom{m+q}{m+q-i},
$$
with $(n+i-q+1)_{m-(i-q)}\equiv (n+i+1)_{m-(i-q)}\pmod{q}$ imply that 
$$
\binom{m+n}{m+q-i}\cdot (i+1)_{m-(i-q)} \equiv \binom{m+q}{m+q-i}\cdot (n+i+1)_{m-(i-q)} 
= \binom{m+n+q}{m+q-i}\cdot \frac{(m+q)!}{i!} \pmod{q}.
$$
Hence (\ref{Formula7}) will follow from 
\begin{equation}
\label{Formula8}
\binom{m+n+q}{m+q-i}\cdot \frac{m!}{(i-q)!} \equiv \binom{m+n+q}{m+q-i}\cdot \frac{(m+q)!}{i!}\pmod{q}.
\end{equation}
Since  $m!/(i-q)! = (i-q+1)\cdot\ldots\cdot m \equiv (i+1)\cdot\ldots\cdot (m+q) = (m+q)!/i!\pmod{q}$ this 
congruence (\ref{Formula8}) is true and thus our formula (\ref{Formula6}) is proven.  
\end{proof}

\begin{cor}
For all $m,s\in\natu_0$, $q\in\natu$, and $i\in\{q,\ldots,m+q\}$ we have
\begin{equation}
\label{Formula9}
(m+q-i)!\cdot\binom{m}{m+q-i}\cdot\binom{m+s}{m+q-i}\equiv 
(m+q-i)!\cdot\binom{m+q}{m+q-i}\cdot\binom{m+s+q}{m+q-i} \pmod{q}.
\end{equation}
\end{cor}
\begin{proof}
Using notations from lemma 5 and assuming that $p=-q$ and $n=s+q$ with $s\in\natu_0$  we obtain
$$
(n+p+i-q+1)_{m+q-i} = (m+q-i)!\cdot\binom{m+s}{m+q-i}~\mbox{and}~(i-q-p+1)_{m+q-i} = (m+q-i)!\cdot\binom{m+q}{m+q-i}.
$$
Hence we can rewrite lemma 5 as
$$
(m+q-i)!\cdot \binom{m}{m+q-i}\cdot \binom{m+s}{m+q-i} \equiv (m+q-i)!\cdot \binom{m+q}{m+q-i}\cdot 
\binom{m+s+q}{m+q-i} \pmod{q},
$$
which is what we had to show.
\end{proof}
Please notice that, for example if $m=3,~q=6,~s=5$, and $i=6$, then for the binomial factors from (\ref{Formula9}) we have
$$
\binom{3}{3}\cdot\binom{3+5}{3} -\binom{3+6}{3}\cdot\binom{3+6+5}{3} \equiv 2\pmod{6},
$$
so the factor $(m+q-i)!$ in (\ref{Formula9}) is necessary to guarantee the equality. 
Now we are ready to give a direct proof of Carlitz' identity for the classical Laguerre polynomials.

\begin{cor}(see \cite{Carlitz})
For all $n,i\in\natu_0$ and $p\in\natu$ the following congruence holds.
$$
(n+p)!L_{n+p}^i(x) \equiv n!L_n^i(x)\cdot p!L_p^i(x) \pmod{p}
$$
\end{cor}
\begin{proof}
First observe that $p!\cdot L_{p}^{i}(x)\equiv (-1)^{p} x^{p} \pmod{p}$ (cf. with (4.6) of \cite{Carlitz}). 
Indeed, $p!\cdot L_p^i(x) = $
$$
=\sum_{t=0}^p (-1)^t \cdot \frac{p!}{t!}\cdot \binom{p+i}{p-t}\cdot x^t = 
p\cdot \sum_{t=0}^{p-1}(-1)^t \cdot (t+1)_{p-1-t}\cdot \binom{p+i}{p-t}\cdot x^t + (-1)^{p} x^p \equiv (-1)^p x^p\pmod{p}.
$$
Therefore it's enough to show that
\begin{equation}
\label{Formula10}
(n+p)!L_{n+p}^i(x) \equiv  (-1)^p x^p \cdot n! L_n^i (x) \pmod{p}.
\end{equation}
We do it by comparing the coefficients of $x^t$ on both sides of (\ref{Formula10}). Suppose first that  $t\in \{0,\ldots, p-1\}$, then the 
coefficient on the R.H.S. of (\ref{Formula10}) is zero. The corresponding coefficient on the L.H.S. of (\ref{Formula10}) is 
$$
(-1)^{t}\cdot \frac{(n+p)!}{t!}\binom{n+p+i}{n+p-t} =(-1)^{t}\cdot \binom{n+p+i}{n+p-t}\cdot \frac{(p-1)!}{t!}\cdot p\cdot (p+1)\cdot \ldots \cdot (n+p) \equiv 0\pmod{p}.
$$
Assume now that $t\in\{p,\ldots,n+p\}$. Using (\ref{LaguerreX}) again we see that the coefficients of $x^t$ on the 
left and right hand sides of (\ref{Formula10}) are respectively 
\begin{equation}
\label{Formula11}
(-1)^{t}\cdot \frac{(n+p)!}{t!}\cdot \binom{n+p+i}{n+p-t}~~~~~\mbox{and}~~~~~
(-1)^{p}\cdot (-1)^{t-p}\cdot \frac{n!}{(t-p)!}\cdot \binom{n+i}{n+p-t}.
\end{equation}
Canceling $(-1)^t$ on both sides we can rewrite these coefficients as 
$$
(n+p-t)!\cdot \binom{n+p}{n+p-t}\cdot \binom{n+p+i}{n+p-t}~~~\mbox{and} ~~~
(n+p-t)!\cdot \binom{n}{n+p-t}\cdot \binom{n+i}{n+p-t}.
$$
Applying corollary 2 with $n=m,~p=q,~t=i$, and $i=s$ we deduce that these coefficients are congruent $\pmod{p}$. 
This finishes our proof of the identity (\ref{Carlitz}).
\end{proof}

\section{Main theorem}

\begin{thm}
For all $n,m\in\natu_0$ and $p,q\in\natu$ we have (cf. with (\ref{Carlitz}) above)
\begin{equation}
\label{Formula12}
 (n+p)!(m+q)! \cdot L_{n+p,m+q}(x,y)\equiv n!m! \cdot L_{n,m}(x,y)\cdot p!q! \cdot L_{p,q}(x,y) \pmod{(p,q)}.
\end{equation}
\end{thm}
\begin{proof}
We will compare the corresponding coefficients of $x^t\cdot y^i$ on both sides of the congruence.
Similarly to the one variable case we have $p!q!\cdot L_{p,q}(x,y)\equiv(-1)^{p+q} \cdot x^p y^q \pmod{(p,q)}$. 
Indeed, using our third formula in (\ref{LagXY}) we have 
$$
p!q!\cdot L_{p,q}(x,y)=\sum_{i=0}^{q}\sum_{t=0}^{p} \frac{(-1)^{i+t}\cdot p!q!}{i! t!}\cdot
\binom{p+q}{q-i}\binom{p+i}{p-t} \cdot x^t y^i,
$$
so if $i+t<p+q$ then $\GCD(p,q)~|~p!q!/i!t!$ since either $i<q$ or $t<p$. If $i=q$ and $t=p$, the coefficient of $x^p y^q$ 
equals $(-1)^{p+q}$ and hence
\begin{equation}
\label{Formula13}
n!m!\cdot L_{n,m}(x,y) \cdot p!q!\cdot L_{p,q}(x,y) \equiv (-1)^{p+q} \cdot x^p y^q\cdot n!m!\cdot L_{n,m}(x,y) \pmod{(p,q)}.
\end{equation}
If we take the coefficient of $x^t y^i$ in $(n+p)!(m+q)!\cdot L_{n+p,m+q}(x,y)$
with $t<p$ or $i<q$ we will have 
$$
(-1)^{i+t}\frac{(n+p)!(m+q)!}{t! i!}\cdot \binom{m+q+n+p}{m+q-i}\cdot \binom{n+p+i}{n+p-t},
$$
and if, for example, $t<p$ then the integer $\frac{(n+p)!}{t!}$ is divisible by $p$, and hence by $\GCD(p,q)$. 
Since $\frac{(m+q)!}{i!}$, $\binom{m+q+n+p}{m+q-i}$, and $\binom{n+p+i}{n+p-t}$ are all integers, $\GCD(p,q)$ divides 
the coefficient of $x^t y^i$. If $t\geq p$ but $i<q$ the proof is similar since $q~|~\frac{(m+q)!}{i!}$. 
So to prove theorem 2 it is enough to show that the coefficients of $x^t y^i$ on the L.H.S. and the R.H.S. of (\ref{Formula12}) 
are congruent $\pmod{(p,q)}$ for all $p\leq t\leq n+p$ and $q \leq i \leq m+q$.

Thus we assume from now on that $t \in \{p,\ldots ,n+p\}$ and $i \in \{q,\ldots,m+q\}$. According to the first formula from our 
theorem 1, the coefficient of $y^i$ on the L.H.S. of (\ref{Formula12}) is
$$
(-1)^i\cdot\frac{(n+p)!(m+q)!}{i!}\cdot \binom{m+q+n+p}{m+q-i}\cdot L_{n+p}^i (x),
$$
which is, due to the identity (\ref{Carlitz}), 
$$
\equiv (-1)^{i+p} \cdot\frac{(m+q)!}{i!}\cdot \binom{m+q+n+p}{m+q-i}\cdot n!\cdot L_n^i (x) \cdot x^p \pmod{p}.
$$
Now, 
$$
\frac{(m+q)!}{i!}\cdot \binom{m+q+n+p}{m+q-i} = (m+q-i)!\cdot\binom{m+q}{m+q-i}\cdot \binom{m+q+n+p}{m+q-i},
$$
which is according to corollary 2 for $s=n+p$ congruent $\pmod{q}$ to
$$
(m+q-i)!\cdot\binom{m}{m+q-i}\cdot \binom{m+n+p}{m+q-i} = \binom{m}{m+q-i}\cdot (n+p+i-q+1)_{m-(i-q)},
$$
which by lemma 5 is 
$$
\equiv \binom{m+n}{m+q-i}\cdot (i-q-p+1)_{m-(i-q)} \pmod{(p,q)}.
$$
It shows that the coefficient of $y^i$ on the L.H.S. of (\ref{Formula12}) is 
$$
\equiv (-1)^{i+p} \cdot \binom{m+n}{m+q-i}\cdot (i-q-p+1)_{m-(i-q)}\cdot n!\cdot L_n^i (x) \cdot x^p \pmod{(p,q)}.
$$
Furthermore, proposition 1 allows to replace $(i-q-p+1)_{m-(i-q)}\cdot n!\cdot L_n^i(x)$ by $(i-q+1)_{m-(i-q)}\cdot n!\cdot L_n^{i-q}(x)$ 
modulo $(p,q)$, so the coefficient of $y^i$ on the L.H.S. of (\ref{Formula12}) is
\begin{equation}
\label{Formula14}
\equiv (-1)^{i+p} \cdot \binom{m+n}{m+q-i}\cdot (i-q+1)_{m-(i-q)}\cdot n!\cdot L_n^{i-q}(x)\cdot x^p \pmod{(p,q)}.
\end{equation}
Now let $s:=t-p$, then according to (\ref{LaguerreX}), the coefficient of $x^t$ in (\ref{Formula14}) will be
$$
(-1)^{i+p} \cdot \binom{m+n}{m+q-i}\cdot (i-q+1)_{m-(i-q)}\cdot \frac{(-1)^s n!}{s!}\cdot \binom{n+i-q}{n+p-t} = 
$$
$$
 = (-1)^{i+t}\frac{n!}{(t-p)!}\cdot \frac{m!}{(i-q)!}\cdot \binom{m+n}{m+q-i}\cdot \binom{n+i-q}{n+p-t} .
$$
Since by theorem 1, the coefficient of $x^{t-p} y^{i-q}$ in $(-1)^{p+q}\cdot n!m!\cdot L_{n,m}(x,y)$ is
$$
(-1)^{p+q}\cdot\frac{(-1)^{t+i-p-q}n!m!}{(t-p)!(i-q)!}\cdot\binom{m+n}{m+q-i}\cdot\binom{n+i-q}{n+p-t},
$$
we see that the coefficients of $x^t y^i$ on the L.H.S. of (\ref{Formula12}) and on the R.H.S. of (\ref{Formula13}) are congruent 
$\pmod{(p,q)}$, and thus our theorem is proven. 
\end{proof}

\section{Irreducibility and other related questions}

First we observe that 
\begin{lem}
For all $n,m\in\natu_0$ the polynomials $L_{n,m}(x,y)$ are irreducible over the rationals. 
\end{lem}
\begin{proof}
Suppose $n\cdot m\neq 0$ and $L_{n,m}(x,y)=f(x,y)\cdot g(x,y)$ with $\deg f(x,y)>0$ and $\deg g(x,y)>0$.
Then according to our Note above, $f(x,x)\cdot g(x,x) = L_{n,m}(x,x)=\binom{n+m}{n}\cdot L_{n+m}(x)$. 
Since $L_k(x)$ is irreducible for all $k\in \natu$ (see \cite{Schur} 
or \cite{Filaseta}), we must have either $\deg f(x,x)=0$ or $\deg g(x,x)=0$. Assuming without loss of generality that $\deg f(x,x)=0$ 
we get $\deg g(x,x) = \deg L_{n+m}(x)=n+m$. Since $\deg g(x,x)\leq \deg g(x,y)$ we deduce from last equality that 
$\deg L_{n,m}(x,y) = n+m\leq \deg g(x,y)$, which contradicts the assumption that $\deg f(x,y)>0$. If $n=0$ or $m=0$ our 
polynomial $L_{n,m}(x,y)$ reduces to the classical one in a single variable, which is irreducible.
\end{proof}

As we have mentioned in the introduction, this is the main distinction of our version of two-variable Laguerre polynomials from those 
considered in \cite{Aktas} and \cite{Dunkl}. As for the orthogonality, one can easily check several examples to see that over the 
domain $\real^2_+=\{(x,y)~|~0<x<\infty,~0<y<\infty\}$ the polynomials $L_{n,m}(x,y)$ are 
neither orthogonal with respect to the weight function $e^{-(x+y)}$ nor with respect to $e^{-(x+y)/2}$. It would be interesting to see 
if $L_{n,m}(x,y)$ are orthogonal with respect to some other weight function.

Laguerre polynomials in one variable have many interesting combinatorial properties. For example, Even and Gillis \cite{Even} in 
1976, showed that an integral of a product of the Laguerre polynomials and $e^{-x}$ can be interpreted as certain 
permutations of a set of objects of different ``colors'' ({\it derangements}).  Jackson \cite{Jackson} in the same year gave a shorter 
proof of this result of Even and Gillis using rook polynomials $R_n(x)$.  These polynomials satisfy 
$$
R_n(x) = \sum\limits_{k=0}^{n} r_k\cdot x^k = n! x^n\cdot L_n(-1/x),
$$ 
where $r_k$ stands for the rook number that counts the number of ways of placing $k$ non-attacking rooks on the full $n\times n$ 
board. We would like to close this paper with a general question if $L_{n,m}(x,y)$ have any combinatorial properties 
similar to those of $L_n(x)$. In particular, the two-dimensional rook numbers and their certain properties 
can be generalized to three and higher dimensions (see, for example, \cite{Alayont}), so we ask if 
$$
n!x^n\cdot m!y^m\cdot L_{n,m}(-1/x,-1/y)
$$ 
have a natural interpretation in terms of rook numbers for three-dimensional boards.

\noindent Siena College, Department of Mathematics\\
515 Loudon Road, Loudonville NY 12211\\ ~ \\
nkrylov@siena.edu and z31li@siena.edu

\end{document}